\documentclass{article}

\usepackage{indentfirst}
\usepackage{amsmath}
\usepackage{amsthm}
\usepackage{amsfonts}
\usepackage{url}

\theoremstyle{plain}
\newtheorem{thm}{Theorem}[section]
\newtheorem{prop}[thm]{Proposition}
\newtheorem{lem}[thm]{Lemma}
\newtheorem{cor}[thm]{Corollary}
\newtheorem{conj}[thm]{Conjecture}

\newtheorem{rmk}[thm]{Remark}

\DeclareMathOperator{\codim}{codim}
\DeclareMathOperator{\Sing}{Sing}
\DeclareMathOperator{\Bir}{Bir}
\DeclareMathOperator{\Aut}{Aut}
\DeclareMathOperator{\id}{id}

\title{On endomorphisms of varieties which are injective on open subsets}
\author{Takumi Asano}
\date{}

\begin{document}

\maketitle

\begin{abstract}
We consider conditions under which endomorphisms of varieties become automorphisms. For example, there is a remarkable theorem, called Ax-Grothendieck theorem, which states that any injective endomorphism of a variety is bijective. Over an algebraically closed field of characteristic zero, bijectivity of endomorphisms of varieties implies that the endomorphisms are automorphisms, thus Ax-Grothendieck theorem gives one of the conditions we considering. There is also a conjecture, called Miyanishi conjecture, which claims that for any endomorphism of a variety over an algebraically closed field of characteristic zero, if it is injective outside a closed subset of codimension at least $2$, then it is an automorphism. Recently, I. Biswas and N. Das prove that any endomorphism which satisfies the conditions of Miyanishi conjecture induces an automorphism of the singular locus of the variety with some conditions. In this paper, we prove that Miyanishi conjecture holds for any threefold which satisfies the conditions of I. Biswas and N. Das. For higher dimensional varieties, we also observe how divisorial contractions affect endomorphisms by using minimal model program theory. We can prove that Miyanishi conjecture holds for any open subset of a projective variety which has a sequence of divisorial contractions to the canonical model or a birationally superrigid Mori fiber space.
\end{abstract}

\section{Introduction}
We work over an algebraically closed field $k$ of characteristic zero, and varieties over $k$ mean separated integral schemes of finite type over $k$. We call $1$-dimensional varieties curves.

J. Ax proved the following theorem in \cite{Ax}.

\begin{thm}[{\cite[Theorem]{Ax}}]\label{A}
Let $A$ and $B$ be schemes such that $A$ is of finite type over $B$. For any endomorphism $\varphi$ of $A$ over $B$, if $\varphi$ is injective, then $\varphi$ is bijective.
\end{thm}

The following generalization of Theorem \ref{A} is conjectured by M. Miyanishi in \cite{Open}.

\begin{conj}[Miyanishi conjecture, \cite{Open}]\label{M}
Let $\varphi:X \to X$ be an endomorphism of a variety $X$ over $k$, and let $Y$ be a closed subset of $X$ which is of codimension at least $2$. If $\varphi$ is injective on $X \setminus Y$, then $\varphi$ is an automorphism.
\end{conj}

There are several known cases when Conjecture \ref{M} holds. For example, affine or complete case \cite{Kal}, and smooth case \cite{Das}. In \cite{A}, we observe Conjecture \ref{M} for quasi-projective varieties by using minimal model program theory.

I. Biswas and N. Das showed the following theorem in \cite{BD}.

\begin{thm}[{\cite[Theorem 1.3]{BD}}]\label{BD}
Let $\varphi:X \to X$ be an endomorphism of a normal variety $X$ over $k$, and let $Y$ be a closed subset of $X$. We take a positive integer $c$ such that $\codim Y \ge c$, and set $W=X \setminus \Sing X$. Assume that $\varphi$ is injective on $X \setminus Y$, and one of the followings holds.

\begin{enumerate}
\item $c=2$, and $\varphi$ is surjective,
\item $c=2$, and $X$ is $\mathbb{Q}$-factorial,
\item $c=2$, and $X$ is locally a complete intersection,
\item $c=3$.
\end{enumerate}

Then, we have $\varphi(W) \subset W$, and $\varphi |_W:W \to W$ is an isomorphism.
\end{thm}

We prove the following by using Theorem \ref{BD}.

\begin{thm}\label{main1}
Let $\varphi:X \to X$ be an endomorphism of a normal variety $X$ over $k$, and let $Y$ be a closed subset of $X$. Assume that $\varphi$ is injective on $X \setminus Y$, $\codim Y=2$, and one of the followings holds.

\begin{enumerate}
\item $\varphi$ is surjective,
\item $X$ is $\mathbb{Q}$-factorial,
\item $X$ is locally a complete intersection.
\end{enumerate}

Then, we can replace $Y$ to make $\codim Y \ge 3$ with the condition that $\varphi$ is isomorphism on $X \setminus Y$. 
\end{thm}

From Theorem \ref{main1}, we can prove that Conjecture \ref{M} holds for threefolds with a certain condition as follows.

\begin{cor}\label{3fld}
In the settings of Theorem \ref{BD}, suppose that $\dim X=3$, and one of the conditions in Theorem \ref{BD} holds. Then, Conjecture \ref{M} holds for $X$.
\end{cor}

By using our argument to prove Theorem \ref{main1}, we can prove the following.

\begin{thm}\label{main2}
Let $\varphi:X \to X$ be an endomorphism of a normal variety $X$ over $k$, and let $Y$ be a closed subset of $X$. Assume that $\varphi$ is injective on $X \setminus Y$, $\codim Y \ge2$, $\dim \Sing(X) \le 1$, and one of the followings holds.

\begin{enumerate}
\item $\varphi$ is surjective,
\item $X$ is $\mathbb{Q}$-factorial,
\item $X$ is locally a complete intersection,
\item $\codim Y \ge 3$.
\end{enumerate}

Then, $\varphi$ is an automorphism.
\end{thm}

We also observe higher dimensional cases by using minimal model program theory, and obtain the following theorem. This generalizes \cite[Theorem 1.4]{A}.

\begin{thm}\label{main3}
Let $X$ be an open subset of a $\mathbb{Q}$-factorial normal projective variety $\overline{X}$ over $k$. Suppose that $\overline{X}$ has canonical singularities and $\overline{X}$ has a sequence of divisorial contractions $\overline{X}=\overline{X}_0 \to \overline{X}_1 \to \cdots \to \overline{X}_n$. If $\Bir(\overline{X}_n)=\Aut(\overline{X}_n)$ holds, then Conjecture \ref{M} holds for $X$, where $\Bir(\overline{X}_n)$ is the birational automorphism group of $\overline{X}_n$, and $\Aut(\overline{X}_n)$ is the automorphism group of $\overline{X}_n$.

In particular, if $\overline{X}_n$ is the canonical model of $\overline{X}$ or a birationally superrigid Mori fiber space, then Conjecture \ref{M} holds for $X$.
\end{thm}

In the rest of this section, we introduce some useful results which are already known. These are the same as in \cite{A}.

S. Kaliman used the following two useful lemmas in \cite{Kal}.

\begin{lem}[{\cite[Lemma 2]{Kal}}]\label{normal}
If Conjecture \ref{M} holds for the normalization of $X$, then it holds for $X$. 
\end{lem}

\begin{lem}[{\cite[Lemma 3]{Kal}}]\label{Z}
Let $Z=X \setminus \varphi(X \setminus Y)$. Then, $Z$ is a closed subset of $X$ and $\dim Y=\dim Z$ if $Z \neq \emptyset$.
\end{lem}

N. Das used the following lemma in \cite{Das}.

\begin{lem}[{\cite[Lemma 2.1]{Das}}]\label{bir1}
Any endomorphism which satisfies the conditions of Conjecture \ref{M} is birational.
\end{lem}

\begin{rmk}\label{bir2}
Actually, the proof of Lemma \ref{bir1} in \cite{Das} implies that any injective dominant morphism of varieties over $k$ is birational.
\end{rmk}

Due to Remark \ref{bir2}, in many situations, we can use a variant of Zariski's main theorem as follows.

\begin{thm}[{\cite[{\S}9, Original form]{Mum}}]\label{ZMT}
Let $f:V \to U$ be a birational morphism of varieties over $k$. If $f$ has finite fibers and $U$ is normal, then $f$ is an open embedding.
\end{thm}

Combining Remark \ref{bir2} and Theorem \ref{ZMT}, we have the following proposition.

\begin{prop}[{\cite[Proposition 1.11]{A}}]\label{bijiso}
Let $f:V \to U$ be a bijective morphism of varieties over $k$. If $U$ is normal, then $f$ is an isomorphism.
\end{prop}

\begin{rmk}\label{finfib}
Since we may assume that $X$ is normal by Lemma \ref{normal}, it is enough to show that $\varphi$ is bijective by Proposition \ref{bijiso}. Moreover, combining Lemma \ref{bir1} and Theorem \ref{ZMT}, we may show that $\varphi$ has finite fibers to prove that $\varphi$ is an automorphism.
\end{rmk}

\section{Dimensional argument}
\subsection{A lemma and some easy cases}
We begin by proving a useful lemma for considering Conjecture \ref{M}.

\begin{lem}\label{YZSing}
Suppose one of the conditions in Theorem \ref{BD} holds. Then, we can replace $Y$ to make $Y,\ Z \subset \Sing(X)$ with the condition that $\varphi$ is injective on $X \setminus Y$.
\end{lem}

To prove Lemma \ref{YZSing}, we use the following fact.

\begin{thm}[{\cite[III. 9. Proposition 1]{Mum}}]\label{ZVP}
Let $X$ be a factorial variety over $k$, and $f:X^{\prime} \to X$ be a birational morphism from a variety $X^{\prime}$ over $k$. Then, there exists an open subset $U$ of $X$ with the following conditions.

\begin{enumerate}
\item $f|_{f^{-1}(U)}:f^{-1}(U) \to U$ is an isomorphism,
\item for any irreducible component $E$ of $X^{\prime} \setminus f^{-1}(U)$, we have $\dim E=\dim X-1$, and $\dim \overline{f(E)} \le \dim X -2$.
\end{enumerate}
\end{thm}

\begin{proof}[Proof of Lemma \ref{YZSing}]
Let $A$ be the set of $x \in X$ such that the irreducible component of $\varphi^{-1}(\{\varphi(x)\})$ which contains $x$ is of dimension at least $1$. Then, we have $A \subset Y$, $A$ is a closed subset of $X$ by \cite[Exercise II. 3.22(d)]{Har}, and $\varphi|_{X \setminus A}:X \setminus A \to X$ has finite fibers.  By Theorem \ref{ZMT}, $\varphi|_{X \setminus A}$ is an open embedding, thus we can replace $Y$ with $A$.

Now, we apply Theorem \ref{ZVP} to $\varphi|_{\varphi^{-1}(W)}:\varphi^{-1}(W) \to W$. Since $\varphi|_{\varphi^{-1}(W)}$ does not contract any divisor, we have $\varphi(Y) \subset \Sing(X)$. By Theorem \ref{BD}, we know $\varphi(W) \subset W$, thus we have $Y \subset \Sing(X)$. Also, this implies $Z \subset \Sing(X)$.
\end{proof}

By Lemma \ref{YZSing}, we find it easy to prove Conjecture \ref{M} in some cases as follows.

\begin{prop}\label{dimeasy}
If one of the followings holds, then Conjecture \ref{M} holds for $X$.

\begin{enumerate}
\item $\dim \Sing(X) \le 0$,
\item $\dim X \le 2$,
\item $\varphi$ is surjective.
\end{enumerate}
\end{prop}

\begin{proof}
We replace $Y$ with $A$ as in the proof of Lemma \ref{YZSing}.\\
\ \\
$(1)$ By Lemma \ref{YZSing}, we have $\dim Y \le 0$. This implies that $\varphi$ has finite fibers, thus $\varphi$ is an automorphism by Remark \ref{finfib}.\\
\ \\
$(2)$ We may assume $X$ is normal by Lemma \ref{normal}. Then, we have $\codim \Sing(X) \ge 2$. Since we suppose $\dim X \le 2$, we have $\dim \Sing(X) \le 0$, thus $\varphi$ is an automorphism by $(1)$.\\
\ \\
$(3)$ Since we have replaced $Y$ with $A$ as in the proof of Lemma \ref{YZSing}, for all $y$ in $Y$, $\varphi^{-1}(\{\varphi(y)\})$ is of dimension at least $1$. By the definition of $Z$ and surjectivity of $\varphi$, we have $\varphi(Y)=Z$. If $Z$ is not empty, then we have $\dim Y>\dim Z$, but this contradicts Lemma \ref{Z}.
\end{proof}

\subsection{Proof of Theorem \ref{main1}}
Let $\varphi:X \to X$ be an endomorphism of a normal variety $X$ over $k$, and let $Y$ be a closed subset of $X$. Assume that $\varphi$ is injective on $X \setminus Y$, $\codim Y=2$. Also, we assume one of the conditions in Theorem \ref{main1}, thus $\varphi|_W:W \to W$ is an isomorphism by Theorem \ref{BD}, where $W=X \setminus \Sing(X)$. We can replace $Y$ with $A$ as in the proof of Lemma \ref{YZSing}, thus if $\codim A \ge 3$, then there is nothing to prove. From now on, we assume $\codim A=2$, and we replace $Y$ with $A$. Note that this assumption implies $\codim \Sing(X)=2$ by Lemma \ref{YZSing}, and $\codim Z=2$ by Lemma \ref{Z}.

\begin{lem}\label{ZSing}
$\Sing(X) \cap Z$ is of codimension at least $3$.
\end{lem}

\begin{proof}
Let $\pi:\widetilde{X} \to X$ be a resolution of singularities of $X$ such that $\pi|_{\pi^{-1}}(W):\pi^{-1}(W) \to W$ is an isomorphism. It suffices to prove that for any $\pi$-exceptional prime divisor $E$, $\pi(E) \cap Z$ is of codimension at least $3$. If $\pi(E) \not\subset Z$, then the assertion is obvious, thus we may assume $\pi(E) \subset Z$.

We prove $\codim \pi(E) \ge 3$. There is a birational automorphism $\widetilde{\varphi}$ of $\widetilde{X}$ such that $\pi \circ \widetilde{\varphi}=\varphi \circ \pi$. Since $\widetilde{X}$ is smooth, there exists an open dense subset $E_0$ of $E$ such that the inverse birational automorphism $\widetilde{\varphi}^{-1}$ of $\widetilde{\varphi}$ is defined on $E_0$. By $\pi(E_0) \subset \pi(E) \subset Z$, we have $(\pi \circ \widetilde{\varphi}^{-1})(E_0) \subset \varphi^{-1}(Z)=Y$. This implies $\pi(E_0) \subset \varphi(Y)$, and by taking the closures, we have $\pi(E) \subset \overline{\pi(E_0)} \subset \overline{\varphi(Y)}$. Since we have replaced $Y$ with $A$, $\overline{\varphi(Y)}$ is of codimension more than $\codim Z=2$. Hence, we have $\codim \pi(E) \ge 3$.
\end{proof}

Let $S_1, \cdots, S_n$ be the irreducible components of $\Sing(X)$ of codimension $2$. By Lemma \ref{ZSing}, we have $S_1, \cdots, S_n \not\subset Z$, thus for all $i \in \{1, \cdots, n\}$, there uniquely exists $j \in \{1, \cdots, n\}$ such that $\overline{\varphi(S_j)}=S_i$. Note that uniqueness of $j$ is clear by the injectivity of $\varphi$ on $X \setminus Y$. In particular, $S_1, \cdots, S_n$ are not contracted by $\varphi$, thus they are not contained in $Y$. By Lemma \ref{YZSing}, this implies $\codim Y > \codim \Sing(X)=2$, and this is a contradiction.

\subsection{Proof of Corollary \ref{3fld}}
By Lemma \ref{YZSing} and Theorem \ref{main1}, we may assume $\codim Y \ge 3$. This implies $\dim Y \le 0$ since we consider a threefold $X$, thus $\varphi$ has finite fibers. Hence $\varphi$ is an automorphism by Remark \ref{finfib}.

\subsection{Proof of Theorem \ref{main2}}
We replace $Y$ with $A$ as in the proof of Lemma \ref{YZSing}, and we take $\pi:\widetilde{X} \to X$ as in the proof of Theorem \ref{main1}. By Lemma \ref{YZSing}, we have $\dim Y \le \dim \Sing(X) \le 1$. If $\dim Y \le 0$, then $\varphi$ has finite fibers, thus $\varphi$ is an automorphism by Remark \ref{finfib}. If $\dim Y=1$, then we prove the following.

\begin{lem}\label{ZSing2}
Let $E$ be a $\pi$-exceptional prime divisor. If $\pi(E)$ is contained in $Z$, then $\pi(E)$ is a point.
\end{lem}

\begin{proof}
It is the same as the proof of Lemma \ref{ZSing}. Note that $\dim Y=1$ implies that $Y$ is the union of all curves on $X$ which are contracted by $\varphi$.
\end{proof}

By Lemma \ref{ZSing2}, we deduce that there are no $\pi$-exceptional prime divisors mapping to a $1$-dimensional irreducible component of $Z$. However, this contradicts $Z \subset \Sing(X)$ since $\pi|_{\pi^{-1}(W)}:\pi^{-1}(W) \to W$ is an isomorphism.

\section{Birational algebraic argument}

\subsection{Some easy cases}
There are some cases that we can deduce easily Conjecture \ref{M} holds. They may seem to be trivial, but they give us some ideas for proof in more general cases. One of the most trivial cases is the following.

\begin{prop}\label{bsrigid}
Let $\overline{X}$ be a complete variety, and $X$ be an open subset of $\overline{X}$. If $\Bir(\overline{X})=\Aut(\overline{X})$ holds, then Conjecture \ref{M} holds for $X$.

In particular, Conjecture \ref{M} holds for $X$ which is an open subset of a canonical model, or a birationally superrigid Mori fiber space.
\end{prop}

\begin{proof}
We can regard the endomorphism $\varphi$ of $X$ as a birational automorphism $\overline{\varphi}$ of $\overline{X}$. By assumption, any birational automorphism of $\overline{X}$ is an automorphism of $\overline{X}$, thus $\overline{\varphi}$ is an automorphism of $\overline{X}$. In particular, $\varphi$ is an automorphism of $X$.
\end{proof}

In the conditions of Theorem \ref{BD}, the first case that $\varphi$ is surjective is treated in Lemma \ref{dimeasy}$(3)$, and the second case that $X$ is locally a complete intersection is treated in \cite{BD}, \cite{Das}. As for the fourth case that $\codim Y \ge 3$, we may think that this case is essential by Theorem \ref{main1}. Thus, we mainly treat the third case that $X$ is $\mathbb{Q}$-factorial in this section. When we consider $\mathbb{Q}$-factorial normal projective varieties, intersection theory is a very powerful tool and gives us the following.

\begin{prop}[{\cite[Theorem 1.3]{A}}]\label{ample}
Let $\overline{X}$ be a $\mathbb{Q}$-factorial normal projective variety, and $X$ be an open subset of $\overline{X}$ with $\codim(\overline{X} \setminus X) \ge 2$. If the canonical divisor $K_{\overline{X}}$ of $\overline{X}$ is ample or anti-ample, then Conjecture \ref{M} holds for $X$.

In particular, Conjecture \ref{M} holds for $X$ which is an open subset of a Fano variety.
\end{prop}

\begin{proof}
We take a resolution of indeterminacy of $\overline{\varphi}$, which is a triple $(\widetilde{X},\ p,\ q)$ such that $\widetilde{X}$ is a smooth projective variety, and $p,\ q$ are birational projective morphisms $\widetilde{X} \to \overline{X}$ which satisfy $q=\overline{\varphi} \circ p$. Note that $(\widetilde{X},\ p,\ q)$ is also a resolution of indeterminacy of $\overline{\varphi}^{-1}$. By Lemma \ref{Z} and the assumption that $\codim(\overline{X} \setminus X) \ge 2$, $\overline{\varphi}$ and $\overline{\varphi}^{-1}$ do not contract any divisor. This implies that the set of $p$-exceptional prime divisors coincides with the set of $q$-exceptional prime divisors.

Suppose that $\varphi$ is not an automorphism. Then, by Remark \ref{finfib}, there is a curve $C$ on $X$ such that $\varphi(C)$ is a point. Let $\overline{C}$ be the closure of $C$ in $\overline{X}$, and $\widetilde{C}$ be a curve on $\widetilde{X}$ such that $p(\widetilde{C})=\overline{C}$. Note that $q(\widetilde{C})$ is a point. By pulling back $K_{\overline{X}}$ by $p$ and $q$ respectively, we have the following equations in the N\'{e}ron-Severi group of $\widetilde{X}$.
$$K_{\widetilde{X}}=p^{\ast}K_{\overline{X}}+\sum_i a(\widetilde{D}_i,\ \overline{X})\widetilde{D}_i,$$
$$K_{\widetilde{X}}=q^{\ast}K_{\overline{X}}+\sum_j a(\widetilde{E}_j,\ \overline{X})\widetilde{E}_j.$$
In these equations, $\widetilde{D}_i$ runs over all $p$-exceptional prime divisors, and $\widetilde{E}_j$ runs over all $q$-exceptional prime divisors. Also note that $a(\widetilde{D}_i,\ \overline{X})$ is a rational number called the discrepancy of $\widetilde{D}_i$ over $\overline{X}$. Since the set of $p$-exceptional prime divisors coincides with the set of $q$-exceptional prime divisors, we have $p^{\ast}K_{\overline{X}}=q^{\ast}K_{\overline{X}}$ in the N\'{e}ron-Severi group of $\widetilde{X}$. This gives the following calculation of the intersection number $(K_{\overline{X}},\ \overline{C})$.
$$(K_{\overline{X}},\ \overline{C})=(K_{\overline{X}},\ p_{\ast}\widetilde{C})=(p^{\ast}K_{\overline{X}},\ \widetilde{C})=(q^{\ast}K_{\overline{X}},\ \widetilde{C})=(K_{\overline{X}},\ q_{\ast}\widetilde{C})=0.$$
This contradicts that $K_{\overline{X}}$ is ample or anti-ample.
\end{proof}

\subsection{Proof of Theorem \ref{main3}}
Let $X$ be an open subset of a $\mathbb{Q}$-factorial normal projective variety $\overline{X}$ over $k$. Suppose that $\overline{X}$ has canonical singularities and $\overline{X}$ has a sequence of divisorial contractions $\overline{X}=\overline{X}_0 \to \overline{X}_1 \to \cdots \to \overline{X}_n$. Let $f$ be the morphism $\overline{X}_0 \to \overline{X}_n$, and we set
$$EPD(f)_X=\{an\ exceptional\ prime\ divisor\ \overline{E}\ of\ f\ |\ \overline{E} \cap X \neq \emptyset\}.$$
We regard the endomorphism $\varphi$ of $X$ as a birational automorphism $\overline{\varphi}$ of $\overline{X}$. We also suppose $\Bir(\overline{X}_n)=\Aut(\overline{X}_n)$, thus for any $\overline{E} \in EPD(f)_X$, its strict transformation $\overline{\varphi}_{\ast}\overline{E}$ by $\overline{\varphi}$ is an element of $EPD(f)_X$ again.

Let $\overline{D}_1, \cdots, \overline{D}_m$ be the prime divisors which appear in the irreducible decomposition of $\overline{X} \setminus X$. Note that if $\codim(\overline{X} \setminus X)$ is at least $2$, then we may ignore this notation since such a prime divisor does not exist. 

We take a resolution of indeterminacy $(\widetilde{X},\ p,\ q)$ of $\overline{\varphi}$. As in the proof of Proposition \ref{ample}, $(\widetilde{X},\ p,\ q)$ is also a resolution of indeterminacy of $\overline{\varphi}^{-1}$. Note that the set of $p$-exceptional prime divisors does not necessarily coincides with the set of $q$-exceptional prime divisors since $\overline{\varphi}$ and $\overline{\varphi}^{-1}$ may contract some divisors which are contained in $\overline{X} \setminus X$. However, we can prove the following lemmas.

\begin{lem}\label{Epq}
Let $p^{-1}_{\ast}\overline{D}_i$ and $q^{-1}_{\ast}\overline{D}_i$ denote the birational transformations of $\overline{D}_i$ by $p$ and $q$ respectively. We set
$$P=\{an\ exceptional\ prime\ divisor\ of\ p\} \setminus \{q^{-1}_{\ast}\overline{D}_1, \cdots, q^{-1}_{\ast}\overline{D}_m\},$$
$$Q=\{an\ exceptional\ prime\ divisor\ of\ q\} \setminus \{p^{-1}_{\ast}\overline{D}_1, \cdots, p^{-1}_{\ast}\overline{D}_m\}.$$
Then, we have $P=Q$.
\end{lem}

\begin{proof}
First of all, we explain the definitions of $P$ and $Q$. If $\overline{\varphi}$ contracts some $\overline{D}_i$, then $q$ must contract $p^{-1}_{\ast}\overline{D}_i$ by $q=\overline{\varphi} \circ p$. Thus, the set of $q$-exceptional prime divisors contains birational transformations by $p$ of prime divisors which are contracted by $\overline{\varphi}$. Birational transformations of prime divisors by $p$ cannot be $p$-exceptional prime divisors, thus we have defined $Q$ as above, and it is similar for $P$.

Take $\widetilde{D} \in P$, and suppose it is not a $q$-exceptional prime divisor. Then, $q(\widetilde{D})$ is a prime divisor intersecting $X$, thus it is not contracted by $\overline{\varphi}^{-1}$. However, this contradicts that $\widetilde{D}$ is a $p$-exceptional prime divisor by $p=\overline{\varphi}^{-1} \circ q$. Since it is clear that $\widetilde{D}$ is not a birational transformation of prime divisor by $p$, we have $P \subset Q$. It is similar for $Q \subset P$.
\end{proof}

\begin{lem}\label{Ephi}
Let $A$ be the set of $\overline{\varphi}$-exceptional prime divisors, and $B$ be the set of $\overline{\varphi}^{-1}$-exceptional prime divisors. Then, we have \#A=\#B.
\end{lem}

\begin{proof}
Note that $A$ and $B$ is contained in $\{\overline{D}_1, \cdots, \overline{D}_m\}$. By $\overline{\varphi}^{-1} \circ \overline{\varphi}=\id_{\overline{X}}$, for two distinct elements $\overline{D}, \overline{E}$ of $\{\overline{D}_1, \cdots, \overline{D}_m\} \setminus A$, their strict transformations $\overline{\varphi}_{\ast}\overline{D}$ and $\overline{\varphi}_{\ast}\overline{E}$ by $\overline{\varphi}$ are two distinct elements of $\{\overline{D}_1, \cdots, \overline{D}_m\} \setminus B$. This implies $\#A \ge \#B$. It is similar for $\#B \ge \#A$.
\end{proof}

For easy notation, we may suppose $A=\{\overline{D}_1, \cdots, \overline{D}_k\}$. Note that if $A=\emptyset$, then we set $k=0$. Thus, we have
$$\{an\ exceptional\ prime\ divisor\ of\ q\} \setminus Q=\{p^{-1}_{\ast}\overline{D}_1, \cdots, p^{-1}_{\ast}\overline{D}_k\}.$$
Also, we can write
$$\{an\ exceptional\ prime\ divisor\ of\ p\} \setminus P=\{q^{-1}_{\ast}\overline{E}_1, \cdots, q^{-1}_{\ast}\overline{E}_k\},$$
where $\{\overline{E}_1, \cdots, \overline{E}_k\} \subset \{\overline{D}_1, \cdots, \overline{D}_m\}$.

By pulling back $K_{\overline{X}}$ by $p$ and $q$, we have the following equations in the N\'{e}ron-Severi group of $\widetilde{X}$.
$$K_{\widetilde{X}}=p^{\ast}K_{\overline{X}}+\sum_i a(\widetilde{D}_i,\ \overline{X})\widetilde{D}_i,$$
$$K_{\widetilde{X}}=q^{\ast}K_{\overline{X}}+\sum_j a(\widetilde{E}_j,\ \overline{X})\widetilde{E}_j.$$
In these equations, $\widetilde{D}_i$ runs over all $p$-exceptional prime divisors, and $\widetilde{E}_j$ runs over all $q$-exceptional prime divisors. By Lemma \ref{Epq}, we have
$$p^{\ast}K_{\overline{X}}-q^{\ast}K_{\overline{X}}=\sum_{j=1}^k a(p^{-1}_{\ast}\overline{D}_j,\ \overline{X})p^{-1}_{\ast}\overline{D}_j-\sum_{i=1}^k a(q^{-1}_{\ast}\overline{E}_i,\ \overline{X})q^{-1}_{\ast}\overline{D}_i$$
in the N\'{e}ron-Severi group of $\widetilde{X}$.

From now on, we suppose that $\varphi$ is not an automorphism. The following propositions are essential for proof of Theorem \ref{main3}.

\begin{prop}\label{nKn}
By Remark \ref{finfib}, there is a curve $C$ on $X$ such that $\varphi(C)$ is a point. Let $\overline{C}$ be the closure of $C$ in $\overline{X}$. Then, we have $(K_{\overline{X}},\ \overline{C}) \ge 0$.
\end{prop}

\begin{proof}
Let $\widetilde{C}$ be a curve on $\widetilde{X}$ such that $p(\widetilde{C})=\overline{C}$. Note that $q(\widetilde{C})$ is a point of $X$. Since we assume that $\overline{X}$ has canonical singularities, and we have $(p^{\ast}K_{\overline{X}},\ \widetilde{C})=(K_{\overline{X}},\ \overline{C})$ and $(q^{\ast}K_{\overline{X}},\ \widetilde{C})=0$ by projection formula, it suffices to show $(p^{-1}_{\ast}\overline{D}_j, \widetilde{C}) \ge 0$ and $(q^{-1}_{\ast}\overline{E}_i, \widetilde{C}) \le 0$ for any $i,\ j \in \{1, \cdots, k\}$.

For any $j \in \{1, \cdots, k\}$, since $\overline{C}$ is not contained in $\overline{D}_j$, we have $\widetilde{C} \not\subset p^{-1}_{\ast}\overline{D}_j$. This implies $(p^{-1}_{\ast}\overline{D}_j, \widetilde{C}) \ge 0$.

For any $i \in \{1, \cdots, k\}$, since $q(\widetilde{C}) \cap  \overline{E}_i$ is empty, we have $\widetilde{C} \cap q^{-1}_{\ast}\overline{E}_i=\emptyset$. This implies $(q^{-1}_{\ast}\overline{E}_i, \widetilde{C})=0$. Note that we need only non-positiveness of the intersection numbers here, but we need also vanishing of the intersection numbers later.
\end{proof}

\begin{prop}\label{nEp}
For any $\overline{E} \in EPD(f)_X$, we have $(\overline{E},\ \overline{C}) \le 0$.
\end{prop}

\begin{proof}
By the assumptions of Theorem \ref{main3}, the operation of taking strict transformation by $\overline{\varphi}$ acts $EPD(f)_X$ as a permutation. Thus, we assume that the action is trivial by replacing $\varphi$ with some iterate.

Take $\overline{E} \in EPD(f)_X$. By pulling back $K_{\overline{X}}+\overline{E}$ by $p$ and $q$ respectively, we have the following equations in the N\'{e}ron-Severi group of $\widetilde{X}$.
$$K_{\widetilde{X}}+p^{-1}_{\ast}\overline{E}=p^{\ast}(K_{\overline{X}}+\overline{E})+\sum_i a(\widetilde{D}_i,\ \overline{X},\ \overline{E})\widetilde{D}_i,$$
$$K_{\widetilde{X}}+q^{-1}_{\ast}\overline{E}=q^{\ast}(K_{\overline{X}}+\overline{E})+\sum_j a(\widetilde{E}_j,\ \overline{X},\ \overline{E})\widetilde{E}_j.$$
In these equations, $\widetilde{D}_i$ runs over all $p$-exceptional prime divisors, and $\widetilde{E}_j$ runs over all $q$-exceptional prime divisors. Also note that $a(\widetilde{D}_i,\ \overline{X},\ \overline{E})$ is a rational number called the log discrepancy of $\widetilde{D}_i$ over $(\overline{X},\ \overline{E})$. We have replaced $\varphi$ by some iterate, we may assume $p^{-1}_{\ast}\overline{E}=q^{-1}_{\ast}\overline{E}$, thus we have
$$p^{\ast}(K_{\overline{X}}+\overline{E})-q^{\ast}(K_{\overline{X}}+\overline{E})=\sum_{j=1}^k a(p^{-1}_{\ast}\overline{D}_j,\ \overline{X},\ \overline{E})p^{-1}_{\ast}\overline{D}_j-\sum_{i=1}^k a(q^{-1}_{\ast}\overline{E}_i,\ \overline{X},\ \overline{E})q^{-1}_{\ast}\overline{D}_i.$$
We know $(p^{\ast}(K_{\overline{X}}+\overline{E}),\ \widetilde{C})=(K_{\overline{X}}+\overline{E},\ \overline{C})$, $(q^{\ast}(K_{\overline{X}}+\overline{E}),\ \widetilde{C})=0$, $a(p^{-1}_{\ast}\overline{D}_j,\ \overline{X},\ \overline{E}) \ge 0$, $(p^{-1}_{\ast}\overline{D}_j,\ \widetilde{C}) \ge 0$ for any $j \in \{1, \cdots, k\}$, and $(q^{-1}_{\ast}\overline{E}_i, \widetilde{C})=0$ for any $i \in \{1, \cdots, k\}$, thus we have
$$(K_{\overline{X}}+\overline{E},\ \overline{C})=\sum_{j=1}^k a(p^{-1}_{\ast}\overline{D}_j,\ \overline{X},\ \overline{E})(p^{-1}_{\ast}\overline{D}_j,\ \widetilde{C}).$$
Since \cite[Lemma 2.27]{KM} tells us $a(p^{-1}_{\ast}\overline{D}_j,\ \overline{X},\ \overline{E}) \le a(p^{-1}_{\ast}\overline{D}_j,\ \overline{X})$ for any $j \in \{1, \cdots, k\}$, we have
\begin{eqnarray*}
(K_{\overline{X}}+\overline{E},\ \overline{C})&=&\sum_{j=1}^k a(p^{-1}_{\ast}\overline{D}_j,\ \overline{X},\ \overline{E})(p^{-1}_{\ast}\overline{D}_j,\ \widetilde{C})\\
&\le& \sum_{j=1}^k a(p^{-1}_{\ast}\overline{D}_j,\ \overline{X})(p^{-1}_{\ast}\overline{D}_j,\ \widetilde{C})=(K_{\overline{X}},\ \overline{C}).
\end{eqnarray*}
Thus, we have
$$
(K_{\overline{X}},\ \overline{C}) \ge (K_{\overline{X}}+\overline{E},\ \overline{C})=(K_{\overline{X}},\ \overline{C})+(\overline{E},\ \overline{C}).$$
This implies $(\overline{E},\ \overline{C}) \le 0$.
\end{proof}

\begin{rmk}\label{ind}
We clarify conditions which we exactly need for proving Proposition \ref{nKn} and \ref{nEp}. In fact, we do not need that $\overline{\varphi}$ is defined on a dense open subset of $\overline{C}$. What we need for proving Proposition \ref{nKn} are that $p(\widetilde{C})=\overline{C}$, $q(\widetilde{C})$ is a point, and $\overline{\varphi}$, $\overline{\varphi}^{-1}$ do not contract any divisors intersecting $X$. For proving Proposition \ref{nEp}, we also need that the exceptional prime divisor is invariant for taking strict transformation by some iterate of $\overline{\varphi}$.
\end{rmk}

Let $f_1$ be the morphism $\overline{X}_0 \to \overline{X}_1$, and Let $\overline{E}$ be the exceptional prime divisor of $f_1$. By Proposition \ref{nKn}, $\overline{C}$ is not contracted by $f_1$, thus we can take a curve $\overline{C}_1=f_1(\overline{C})$ in $\overline{X}_1$. We show $(K_{\overline{X}_1},\ \overline{C}_1) \ge 0$.

If $\overline{E}$ is an element of $EPD(f)_X$, then by pulling back $K_{\overline{X}_1}$ by $f_1$, we have
$$K_{\overline{X}_0}=(f_1)^{\ast}K_{\overline{X}_1}+a(\overline{E},\ \overline{X}_1)\overline{E}$$
in the N\'{e}ron-Severi group of $\overline{X}_0$. This implies
$$(K_{\overline{X}_1},\ \overline{C}_1)=(K_{\overline{X}_1},\ (f_1)_{\ast}\overline{C})=((f_1)^{\ast}K_{\overline{X}_1},\ \overline{C})=(K_{\overline{X}_0},\ \overline{C})-a(\overline{E},\ \overline{X}_1)((\overline{E},\ \overline{C}) \ge 0$$
by Proposition \ref{nKn}, \ref{nEp}, and projection formula.

If $\overline{E}$ is not an element of $EPD(f)_X$, then by Proposition \ref{nKn} and Remark \ref{ind}, we can show $(K_{\overline{X}_1},\ \overline{C}_1) \ge 0$.

Thus, we can take a curve $\overline{C}_2=f_2(\overline{C}_1)$ in $\overline{X}_2$, where $f_2$ is the morphism $\overline{X}_1 \to \overline{X}_2$. We repeat this argument again and again, and finally, we can take a curve $\overline{C}_n=f(\overline{C})$ in $\overline{X}_n$. Let $\overline{\varphi}_n$ be the birational automorphism of $\overline{X}_n$ which satisfies $\overline{\varphi}_n \circ f=f \circ \overline{\varphi}$. Since we suppose $\Bir(\overline{X}_n)=\Aut(\overline{X}_n)$, $\overline{\varphi}_n$ is actually an automorphism of $\overline{X}_n$, in particular, it does not contract $\overline{C}_n$. Hence, $\overline{C}$ is not contracted by $\overline{\varphi}_n \circ f$, but this contradicts that $\overline{C}$ is contracted by $f \circ \overline{\varphi}$.

\section*{Acknowledgements}
The author thanks Seidai Yasuda, my supervisor, for his invaluable guidance. The author also thanks Hisato Matsukawa for his advice on the content of this paper. This work was supported by JST SPRING, Grant Number JPMJSP2119.

\end{document}